\documentclass[a4paper,10pt]{amsart}

\usepackage{amsmath}
\usepackage{amssymb}
\usepackage{amsthm}
\usepackage{verbatim}
\usepackage[all]{xy}
\usepackage{enumerate}
\usepackage{dsfont}
\usepackage{color}

\newtheorem{thm}{Theorem}[section]
\newtheorem{prop}[thm]{Proposition}

\newtheorem{lem}[thm]{Lemma}

\theoremstyle{definition}

\newtheorem{dfn}[thm]{Definition}

\newtheorem{rmk}[thm]{Remark}

\numberwithin{equation}{section}

\numberwithin{equation}{section}
\newcommand{\cB}{\mathcal{B}}

\newcommand{\minotimes}{\otimes_{{\rm min}}}

\newcommand{\cA}{\mathcal{A}}

\newcommand{\op}{{\rm op}}

\newcommand{\RH}{H_{\mathbb{R}}}

\title[On the isomorphism class of $q$-Gaussian C$^\ast$-algebras for infinite variables]{On the isomorphism class of $q$-Gaussian C$^\ast$-algebras for infinite variables}

\date{\noindent \today.  \\
 {\it MSC2010}: 46L35, 46L06.   {\it Keywords}: $q$-Gaussian C$^\ast$-algebras, Akemann-Ostrand property.\\
  MC is supported by the NWO Vidi grant VI.Vidi.192.018 `Non-commutative harmonic analysis and rigidity of operator algebras'. MK is supported by the NWO project 613.009.125
`The structure of Hecke-von Neumann algebras'. MW was supported by the European Research Council Starting Grant 677120 INDEX.
 }

\author[Borst]{Matthijs Borst}
\author[Caspers]{Martijn Caspers}
\author[Klisse]{Mario Klisse}

\address{Matthijs Borst, Martijn Caspers, Mario Klisse, TU Delft, EWI/DIAM,
	P.O.Box 5031,
	2600 GA Delft,
	The Netherlands}

\email{M.J.Borst@tudelft.nl}
\email{M.P.T.Caspers@tudelft.nl}
\email{M.Klisse@tudelft.nl}

\author[Wasilewski]{Mateusz Wasilewski}

\address{Mateusz Wasilewski,  IMPAN, \'Sniadeckich 8, 00-656 Warsaw, Poland}

\email{mwasilewski@impan.pl}

\begin{document}

\maketitle

\begin{abstract}
For a real Hilbert space $\RH$ and $-1 < q < 1$ Bozejko and Speicher introduced the C$^\ast$-algebra $A_q(\RH)$ and von Neumann algebra $M_q(\RH)$ of $q$-Gaussian variables. We prove that if $\dim(\RH) = \infty$ and $-1 < q < 1, q \not = 0$ then  $M_q(\RH)$ does not have the Akemann-Ostrand property with respect to $A_q(\RH)$. It follows that $A_q(\RH)$ is not isomorphic to $A_0(\RH)$. This gives an answer to the C$^\ast$-algebraic part  of Question 1.1 and Question 1.2   in \cite{NelsonZeng}.
\end{abstract}

\section{Introduction}

In \cite{BS} Bo\.{z}ejko and Speicher introduced a non-commutative version of Brownian motion using a construction that is now commonly known as the $q$-Gaussian algebra where $-1 \leq q \leq 1$. These algebras range between the extreme Bosonic case $q = 1$ of  fields of classical Gaussian random variables and the Fermionic case $q = -1$ of Clifford algebras. For $q = 0$ one obtains Voiculescu's free Gaussian functor. $q$-Gaussians can be studied on the level of $\ast$-algebras $\cA_q(\RH)$, C$^\ast$-algebras $A_q(\RH)$ and von Neumann algebras $M_q(\RH)$ starting from a real Hilbert space $\RH$ where $\dim(\RH)$ usually refers to the number of variables.

 The dependence of $q$-Gaussian algebras on the parameter $q$ has been an intriguing problem ever since their introduction. The $\ast$-algebras $\cA_q(\RH)$ are easily seen to be isomorphic for all $-1 < q < 1$ (see \cite[Theorem 4.1, proof]{CIW}).  However, the isomorphisms do not extend to the C$^\ast$-algebras $A_q(\RH)$; one way to see this is that this isomorphism maps generators $W_q(\xi)$ to generators $W_{q'}(\xi)$ with $\xi \in \RH$ (see Preliminaries for notation) which is easily seen to be non-isometric.   In fact, the isomorphism problem becomes notoriously difficult on the level of the C$^\ast$-algebras and von Neumann algebras.

A breakthrough result was obtained by Guionnet-Shlyakhtenko in \cite{GS} where free transport techniques were developed to show that in case $\dim(\RH) < \infty$ one has that  $A_q(\RH) \simeq A_0(\RH)$ and $M_q(\RH) \simeq M_0(\RH)$ for a range of $q$ close to 0. The range becomes smaller as $\dim(\RH)$ increases. The proof is also based on the existence and power series estimates of conjugate variables by Dabrowski \cite{Dabrowski}.

The infinite variable case $\dim(\RH) = \infty$ was then pursued by Nelson-Zeng \cite{NelsonZeng} where they explicitly ask whether given a fixed Hilbert space $\RH$ one can have isomorphism of the $q$-Gaussian C$^\ast$- and von Neumann algebras, see \cite[Questions 1.1 and 1.2]{NelsonZeng}. They already note that the condition $q^2 \dim(\RH) < 1$ is required for the construction of conjugate variables to the free difference quotient \cite{Dabrowski}. However, by passing to mixed $q$-Gaussians with sufficient decay on the coefficient array $Q = (q_{ij})_{i,j}$ they show that free transport techniques can still be developed in order to extend the Guionnet-Shlyakhtenko result to this mixed $q$-Gaussian setting. This approach is in some sense sufficiently close to the case of finite dimensional $\RH$. The main merit of the current note is a rather definite and negative answer to the C$^\ast$-algebraic part of \cite[Questions 1.1 and 1.2]{NelsonZeng}, namely we show that we have $A_0(\RH) \not \simeq A_q(\RH), -1 < q < 1, q \not  = 0$ in case the dimension of $\RH$ is infinite.

\vspace{0.3cm}

Our main result is that if $\dim(\RH) = \infty$ then the von Neumann algebra $M_q(\RH)$ does not have the Akemann-Ostrand property with respect to the natural C$^\ast$-subalgebra $A_q(\RH)$ for any $-1 < q < 1, q \not = 0$. This will then distinguish $A_0(\RH)$ from $A_q(\RH)$. The idea of our proof is as follows. In \cite[Theorem 5.1]{Connes} Connes proved that a finite von Neumann algebra $M$ is amenable if and only if the map
\[
M \otimes_{{\rm alg}} M^{\op} \rightarrow \cB(L_2(M)): a \otimes b^{\op} \rightarrow a b^{\op},
\]
is $\minotimes$-bounded. This characterisation -- in combination with a Khintchine inequality -- was used by Nou \cite{Nou} to show that $M_q(\RH)$ is not amenable for $-1 < q < 1$ and $\dim(\RH) \geq 2$.  We show that if $\dim(\RH) = \infty$ and $-1 < q < 1, q \not = 0$ then we cannot even have that
\[
A_q(\RH) \otimes_{{\rm alg}} A_q(\RH)^{\op} \rightarrow \cB(L_2(M_q(\RH))) \slash \mathcal{K}(L_2(M_q(\RH))): a \otimes b^{\op} \rightarrow a b^{\op} + \mathcal{K}(L_2(M_q(\RH))),
\]
is $\minotimes$-bounded where we have taken a quotient by compact operators. This is proved in Section \ref{Sect=AO}. We then harvest the non-isomorphism results in Section \ref{Sect=NonIso}.

\vspace{0.3cm}

\noindent {\bf Acknowledgements:} MK wishes to thank IMPAN where part of this research was carried out during a research visit.

\section{Preliminaries}

\subsection{Von Neumann algebras} In the following $\cB(H)$ denotes the bounded operators on a Hilbert space $H$ and $\mathcal{K}(H)$ denotes the compact operators on $H$.  For a von Neumann algebra $M$ we denote by $(M, L_2(M), J, L_2(M)^+)$ the standard form. For $x \in M$ we write $x^{\op} := J x^\ast J$ which is the right multiplication with $x$ on the standard space. This way $L_2(M)$ becomes an $M$-$M$-bimodule called the trivial bimodule.

The algebraic tensor product is denoted by $\otimes_{{\rm alg}}$   and $\minotimes$ is the minimal tensor product of C$^\ast$-algebras which by   Takesaki's theorem \cite[Theorem IV.4.19]{Takesaki}  is the spatial tensor product.

\subsection{$q$-Gaussians}

Let $-1 < q < 1$.  Now let $\RH$ be a real  Hilbert space with complexification $H := \RH \oplus i \RH$. We define the symmetrization operator $P_q^k$ on $H^{\otimes k}$ by
\begin{equation} \label{Eqn=Pq}
P_q^k(\xi_1 \otimes \ldots \otimes \xi_n) = \sum_{\sigma \in S_k} q^{i(\sigma)} \xi_{\sigma(1)} \otimes \ldots \otimes \xi_{\sigma(n)},
\end{equation}
where $S_k$ is the symmetric group of permutations of $k$ elements and $i(\sigma) := \# \{ (a,b) \mid a < b, \sigma(b) < \sigma(a) \}$ the number of inversions. The operator $P_q^k$ is positive and  invertible \cite{BS}. Define a new inner product on $H^{\otimes k}$ by
\[
\langle \xi, \eta \rangle_q := \langle P_q^k \xi, \eta \rangle,
\]
and call the new Hilbert space $H_q^{\otimes k}$. Set the Hilbert space $F_q(H) := \mathbb{C} \Omega \oplus (\oplus_{k=1}^\infty H_q^{\otimes k})$ where $\Omega$ is a unit vector called the vacuum vector. For $\xi \in H$ let
\[
l_q(\xi) ( \eta_1 \otimes \ldots \otimes \eta_k) := \xi \otimes \eta_1 \otimes \ldots \otimes \eta_k, \qquad l_q(\xi) \Omega = \xi,
\]
and then $l_q^\ast(\xi) = l_q(\xi)^\ast$. These `creation'  and `annihilation' operators are bounded and extend to $F_q(H)$. We define a $\ast$-algebra, C$^\ast$-algebra and von Neumann algebra by
\[
\cA_q(\RH) := \ast{\rm -alg} \{ l_q(\xi) + l_q^\ast(\xi) \mid \xi \in \RH \}, \quad A_q(\RH) := \overline{ \cA_q(\RH) }^{\Vert \cdot  \Vert}, \quad M_q(\RH)  := A_q(\RH)'',
\]
where $\ast$-alg denotes the unital $\ast$-algebra in $\cB(F_q(H))$ generated by the set.   Then $\tau_\Omega(x) := \langle x \Omega, \Omega \rangle$ is a faithful tracial state on $M_q(\RH)$ which is moreover normal. Now $F_q(H)$ is the standard form Hilbert space of $M_q(\RH)$ and $J x \Omega = x^\ast \Omega$.

For $K_{\mathbb{R}}$ a closed subspace of $\RH$ we have that $\cA_q(K_{\mathbb{R}})$ is naturally a $\ast$-subalgebra of  $\cA_q(\RH)$. Further, if $(K_{\mathbb{R},i})_{ i \in \mathbb{N}}$ is an increasing sequence of closed subspaces whose span is dense in  $\RH$ then  $\cup_i \cA_q(K_{\mathbb{R},i})$ is dense in $A_q(\RH)$.

For vectors $\xi_1, \ldots, \xi_k \in H$ there exists a unique operator $W_q(\xi_1 \otimes \ldots \otimes \xi_k) \in \cA_q(\RH)$ such that
\[
W_q(\xi_1 \otimes \ldots \otimes \xi_k) \Omega = \xi_1 \otimes \ldots \otimes \xi_k.
\]
These operators are called Wick operators. It follows that $W_q(\xi)^\op \Omega = \xi$.
We shall further need the constant
\begin{equation}\label{Eqn=Cq}
C_q := \prod_{i=1}^\infty (1 - q^i)^{-1} > 0.
\end{equation}

\section{Main theorem: failure of the Akemann-Ostrand property}\label{Sect=AO}

\subsection{Failure of AO}
We will work with the following definition of the Akemann-Ostrand property \cite{BrownOzawa}.

\begin{dfn}\label{Dfn=AO}
A finite von Neumann algebra $M$ has the Akemann-Ostrand property (or AO) if there exists a $\sigma$-weakly dense unital C$^\ast$-subalgebra $A \subseteq M$ such that $A$ is locally reflexive (see \cite{BrownOzawa}) and such that the multiplication map $\theta: A \otimes_{{\rm alg}} A^{\op} \rightarrow \cB(L_2(M))\slash \mathcal{K}(L_2(M)): a \otimes b^\op \rightarrow a b^\op + \mathcal{K}(L_2(M))$ is continuous with respect to the minimal tensor norm.
We also say that $M$ has AO  with respect to $A$.
\end{dfn}

We assumed local reflexivity of the C$^\ast$-algebra $A$ in Definition \ref{Dfn=AO} as part of the usual definition of AO. However, in the current  paper local reflexivity does not play a crucial role  and all our results hold if we consider Definition  \ref{Dfn=AO} without the  local reflexivity assumption on $A$.

 Note that $\theta$ in Definition \ref{Dfn=AO} is a $\ast$-homomorphism, so if it is continuous it is automatically a contraction.

\begin{thm} \label{Thm=AntiAO}
Let $M$ be a finite von Neumann algebra  with a $\sigma$-weakly dense unital C$^\ast$-subalgebra $A$.
Suppose there exists a unital C$^\ast$-subalgebra $B \subseteq A$ and infinitely many mutually orthogonal closed subspaces $H_i \subseteq L_2(M), i \in \mathbb{N}$ that are left and right $B$-invariant. Suppose moreover that there exists $\delta > 0$ and finitely many operators $b_j, c_j \in B$ such that for every $i \in \mathbb{N}$ we have
\begin{equation}\label{Eqn=NonContraction}
\Vert \sum_{j} b_j c_j^{op} \Vert_{\cB(H_i)}  \geq (1 + \delta)  \Vert \sum_j b_j \otimes c_j^{\op} \Vert_{B \minotimes B^{\op}}.
\end{equation}
Then $M$ does not have AO with respect to $A$.
\end{thm}
\begin{proof}
Since there are infinitely many $B$-$B$-invariant spaces $H_i$ we have for any finite rank operator $x \in \cB(L_2(M))$ that
\[
\Vert \sum_{j} b_j c_j^{op} + x \Vert_{\cB(L_2(M) )  }  \geq (1 + \delta)  \Vert \sum_j b_j \otimes c_j^{\op} \Vert_{B \minotimes B^{\op}}.
\]
Taking the infimum over all such $x$ we obtain that
\begin{equation}\label{Eqn=Estimate1}
\Vert \sum_{j} b_j c_j^{op} +  \mathcal{K}(L_2(M)  )  \Vert_{\cB(L_2(M) ) \slash \mathcal{K}(L_2(M)  )  }  \geq (1 + \delta)  \Vert \sum_j b_j \otimes c_j^{\op} \Vert_{B \minotimes B^{\op}}.
\end{equation}
But the definition of AO entails the existence of a contraction $\theta: A \minotimes A^{\op} \rightarrow \cB( L_2(M)) \slash  \mathcal{K}(L_2(M) )$ such that $\theta(b \otimes c^{\op}) = b c^{\op} +  \mathcal{K}(L_2(M) )$   for all $b,c \in A$. Hence
\[
\Vert \sum_{j} b_j c_j^{op} +  \mathcal{K}(L_2(M)  )  \Vert_{\cB(L_2(M) ) \slash \mathcal{K}(L_2(M) )   }
\leq
  \Vert \sum_j b_j \otimes c_j^{\op} \Vert_{B \minotimes B^{\op}},
\]
which contradicts \eqref{Eqn=Estimate1}.
\end{proof}

\subsection{The case of $q$-Gaussians}

\begin{thm}\label{Thm=AntiAOGaussian}
Assume $\dim(\RH) = \infty$ and $-1 < q < 1, q \not = 0$. Then the von Neumann algebra $M_q(\RH)$ does not have AO with respect to $A_q(\RH)$.
\end{thm}
\begin{proof}
Let $d \geq 2$ be so large that $q^2 d > 1$.  Let
\[
M := M_q(\mathbb{R}^d \oplus \RH), \quad A := A_q(\mathbb{R}^d \oplus \RH),\quad  B := A_q(\mathbb{R}^d \oplus 0).
\]
We shall prove that $M$ does not have AO with respect to $A$; since $\mathbb{R}^d \oplus \RH \simeq \RH$ this suffices to conclude the proof.

Let $\{ f_i \}_i$ be an orthonormal basis of $0 \oplus \RH$. Let $H_{q,i} := \overline{ B f_i B }^{\Vert \cdot \Vert}$ as a closed subspace of the Fock space $F_q(\mathbb{R}^d \oplus \RH)$. Then $H_{q,i} \perp H_{q,j}$ if $i \not = j$ which can be seen straight from the definition of $\langle \cdot, \cdot \rangle_q$. For $k \in \mathbb{N}$ let
\[
\cB(k)  = \{ W_q(\xi) \mid \xi \in (\mathbb{R}^d \oplus 0)^{\otimes k} \}.
\]

Let $\xi, \eta \in (\mathbb{R}^d \oplus 0)^{\otimes k}$ and write $\xi = \xi_1 \otimes \ldots \otimes \xi_k$ with $\xi_i \in \mathbb{R}^ d$.  We have $W_q(\xi)^\ast = W_q(\xi^\ast)$ where $\xi^\ast = \xi_k \otimes \ldots \otimes \xi_1$. We have that
(see \cite{EffrosPopa}),
\[
\langle W_q(\xi) f_i W_q(\eta), f_i \rangle_q
= \langle  f_i W_q(\eta) ,  W_q(\xi)^\ast f_i \rangle_q
 = \langle f_i \otimes \eta, \xi^\ast \otimes f_i \rangle_q = \langle P_q^{k+1} f_i \otimes \eta, \xi^\ast \otimes f_i \rangle.
\]
We examine the right hand side of this expression. The $q$-symmetrization operator $P_q^{k+1}$ is defined as a sum of permutations $\sigma \in S_{k+1}$ (see \eqref{Eqn=Pq}) and it follows from the fact that $f_i \in 0 \oplus \RH$ and $\xi, \eta \in (\mathbb{R}^d \oplus 0)^{\otimes k}$ that the only summands that contribute a possibly non-zero term are the ones where $\sigma(k+1) = 1$. Note that for such a permutation $\sigma$ we have 
\[
i(\sigma) = \# \left( \{ (a, k+1)  \mid 1 \leq a \leq k \} \cup  \{ (a, b)  \mid 1 \leq a <  b \leq k, \sigma(b) < \sigma(a) \} \right).
\]
 Therefore we find,
\begin{equation}\label{Eqn=Conclude}
\begin{split}
\langle W_q(\xi) f_i W_q(\eta), f_i \rangle_q =  &  \sum_{\sigma \in S_k}  q^{k + i(\sigma)} \langle  \eta_{\sigma(1)} \otimes \ldots \otimes \eta_{\sigma(k)}, \xi_k \otimes \ldots \otimes \xi_1  \rangle\\
 = &  q^k \langle P_q^{k} \eta, \xi^\ast  \rangle = q^k \langle \eta, \xi^\ast \rangle_q = q^k \langle W_q(\xi)   \Omega W_q(\eta),  \Omega \rangle_q.
\end{split}
\end{equation}
Now from \eqref{Eqn=Conclude}  we conclude that for $b_j, c_j \in \cB(k)$,
\begin{equation}\label{Eqn=LowerEstimate}
\Vert \sum_j b_j c_j^{\op} \Vert_{\cB(H_{q,i})} \geq
 \vert \langle  \sum_j b_j c_j^{\op}  f_i, f_i  \rangle_q \vert
=
 \vert \sum_j \langle b_j f_i c_j, f_i \rangle_q \vert =
 \vert  \sum_j q^k \langle b_j \Omega c_j, \Omega \rangle_q \vert.
\end{equation}

Now let $\{ e_1, \ldots, e_d \} $ be an orthonormal basis of $\mathbb{R}^d \oplus 0$ and for $j = (j_1, \ldots, j_k) \in \{ 1,\ldots, d \}^k$ let $e_j = e_{j_1} \otimes \ldots \otimes e_{j_k}$.  Let $J_k$ be the set of all such multi-indices of length $k$. So $\# J_k = d^ k$.  Set $\xi_j = (P^k_q)^{- \frac{1}{2} } e_j$ so that $\langle \xi_j, \xi_j \rangle_q = \langle P_q^k \xi_j, \xi_j \rangle = 1$.

Now \eqref{Eqn=LowerEstimate} yields that for all $k \geq 1$ and all $i$,
\[
\begin{split}
\Vert \sum_{j \in J_k}  W_q(\xi_j)^\ast    W_q(\xi_j)^{\op} \Vert_{\cB(H_{q,i})} \geq &  \sum_{j \in J_k} q^k \langle W_q(\xi_j)^\ast \Omega W_q(\xi_j), \Omega \rangle_q =
\sum_{j \in J_k}  q^k \langle \Omega W_q(\xi_j),  W_q(\xi_j) \Omega \rangle_q\\
 = & \sum_{j \in J_k}  q^k  \langle \xi_j,  \xi_j \rangle_q =    q^k d^{k}.
\end{split}
\]
On the other hand from \cite[Proof of Theorem 2]{Nou} we find,
\[
\Vert \sum_{j \in J_k} W_q(\xi_j)^\ast \otimes  W_q(\xi_j)^{\op} \Vert_{B \minotimes B^{\op}} \leq C_q^3 (k+1)^2  d^{k/2},
\]
where the constant $C_q >0$ was defined in \eqref{Eqn=Cq}.
Therefore, as  $q^2 d  > 1$  there exists $\delta > 0$ such that for  $k$ large enough we have for every $i$,
\[
\Vert \sum_{j \in J_k} W_q(\xi_j)^\ast    W_q(\xi_j)^{\op} \Vert_{\cB(H_{q,i} )}  \geq   (1+ \delta)    \Vert \sum_{j \in J_k} W_q(\xi_j)^\ast \otimes  W_q(\xi_j)^{\op} \Vert_{B \minotimes B^{\op}}.
\]
Hence the assumptions of Theorem \ref{Thm=AntiAO} are witnessed which shows that AO does not hold.
\end{proof}

\section{A non-isomorphism result for  $q$-Gaussian C$^\ast$-algebras}\label{Sect=NonIso}
We now turn to the isomorphism question of $A_q(\RH)$ for $q$ close to 0.
 We first need a result of independent interest which seems not to be proved in the literature. By \cite{Ricard} we know that the von Neumann algebra $M_q(\RH)$ with $\dim(\RH) \geq 2$  is a factor of type II$_1$. This was proven already in the case $\dim(\RH) = \infty$ in \cite[Theorem 2.10]{BKS}. In this section we need a strengthening of the latter result, namely that $A_q(\RH)$ has a unique tracial state. The proof is based again on Nou's Khintchine inequality \cite{Nou}.

\begin{thm}\label{Thm=Simple}
Let $\dim(\RH) = \infty$. Then $A_q(\RH)$ has a unique tracial state and is a simple C$^\ast$-algebra.
\end{thm}

We will prove the theorem after first proving a lemma.
 Assume for simplicity that $\RH$ is separable. Let $\{ e_i \}_{i = 1}^\infty$ be an orthonormal basis of $\RH$ and identify $\mathbb{R}^d$ with the span of $\{ e_i \}_{i=1}^d$. For $m \in \mathbb{N}$ consider the map $\cA_q(\RH) \rightarrow \cA_q(\RH)$ given by
\[
\Phi_m(X) =    \frac{1}{m} \sum_{i=1}^m W_q(e_i) X W_q(e_i).
\]
Then $\Phi_m$ extends to a bounded map $A_q(\RH) \rightarrow A_q(\RH)$ with bound uniform in $m$.

The following lemma is stronger than \cite[Theorem 2.10, proof]{BKS} where only weak convergence was established; the result is used in the proof of \cite[Theorem 2.14]{BKS} but its proof is not given.  Therefore we give it here.

\begin{lem} \label{Lem=NormConvergence}
For $X = W_q(\xi), \xi \in H^{\otimes n}$  we have $\Phi_m(X) \rightarrow q^n X$ as $m \rightarrow \infty$ in the norm of $A_q(\RH)$.
\end{lem}
\begin{proof}
First assume that there exists $d \in \mathbb{N}$  such that $\xi \in (\mathbb{C}^{d})^{\otimes n} \subseteq H^{\otimes n}$. By density and uniform boundedness of $\Phi_m$ in $m$ this suffices to conclude the lemma.
 Then, for $m > d$, by \cite[Theorem 3.3]{EffrosPopa},
\begin{equation}\label{Eqn=ThreeSummands}
\Phi_m( W_q(\xi) ) =   \frac{1}{m} \sum_{i=1}^d W_q(e_i) W_q(\xi) W_q(e_i) +   \frac{1}{m} \sum_{i=d+1}^m q^n  W_q(\xi)  +   \frac{1}{m} \sum_{i=d+1}^m  W_q(e_i \otimes \xi \otimes e_i).
\end{equation}
The first term converges to 0 as $m \rightarrow \infty$, whereas the second term converges to $q^n W_q(\xi)$. It thus remains to show that the last term converges to 0 in norm.
We have by \cite[Lemma 2]{Nou} (see also \cite{BozejkoUltra} where a weaker but sufficient estimate was obtained),
\[
\Vert \sum_{i=d+1}^m  W_q(e_i \otimes \xi \otimes e_i) \Vert \leq
(n+3)
C_q^{\frac{3}{2}}
\Vert \sum_{i=d+1}^m   e_i \otimes \xi \otimes e_i  \Vert_{ H_q^{\otimes n+2} }.
\]
The vectors  $\{ e_i \otimes \xi \otimes e_i \}_{i}$ are orthogonal in   $H_q^{\otimes n+2}$  and have the same norm which we denote by $C$. Therefore,
\[
\frac{1}{m}  \Vert \sum_{i=d+1}^m  W_q(e_i \otimes \xi \otimes e_i) \Vert \leq
(n+3)
C_q^{\frac{3}{2}}  C m^{-\frac{1}{2}}.
\]
 We conclude that the third term in \eqref{Eqn=ThreeSummands} converges to 0 as $m \rightarrow \infty$ in norm.
\end{proof}

\begin{proof}[Proof of Theorem \ref{Thm=Simple}]
By Lemma \ref{Lem=NormConvergence} for $X \in \cA_q(\RH)$ set the norm limit $\Phi(X) := \lim_{m \rightarrow \infty} \Phi_m(X)$.
Let $\tau$ be any tracial state on $A_q(\RH)$. Then,  for $X \in \cA_q(\RH)$,
\[
\begin{split}
\tau( \Phi(X) ) = & \lim_{m \rightarrow \infty}  \tau(  \Phi_m(X)) =
\lim_{m \rightarrow \infty}   \frac{1}{m}  \sum_{i=1}^m   \tau( W_q(e_i) X W_q(e_i)   ) \\
 = &
\lim_{m \rightarrow \infty}   \frac{1}{m}  \sum_{i=1}^m   \tau( X W_q(e_i)   W_q(e_i)  )
= \tau(X \Phi(1)) = \tau(X).
\end{split}
\]
Therefore, by  Lemma \ref{Lem=NormConvergence},  $\tau(W_q(\xi)) =   \tau( \Phi^k(W_q(\xi)) ) = q^{kn} \tau(W_q(\xi))$ for  $\xi \in H^{\otimes n}$. For $k \rightarrow \infty$ the expression converges to 0  for $n \geq 1$. It follows that for $X \in \cA_q(\RH)$ we have $\tau(X) = \tau_\Omega(X)$ and by continuity this actually holds for $X \in A_q(\RH)$. So $\tau_\Omega$ is the unique tracial state on $A_q(\RH)$.

Simplicity was already obtained in \cite[Theorem 2.14]{BKS}; it is also based on Lemma \ref{Lem=NormConvergence}.
\end{proof}

The following proposition was also proved in \cite[Chapter 4]{HoudayerThesis}; the proof uses the same method as \cite{Shlyakhtenko} where this result was also obtained for finite dimensional $\RH$.

 \begin{prop} \label{Prop=AO}
For any real Hilbert space $\RH$ the von Neumann algebra  $M_0(\RH)$ satisfies AO with respect to $A_0(\RH)$.
 \end{prop}

\begin{thm}\label{Thm=Main}
Let $\RH$ be a real Hilbert space with $\dim(\RH) = \infty$. Then $A_q(\RH)$ with $-1 < q <  1, q \not = 0$ is not isomorphic to $A_0(\RH)$ and neither to $A_{q'}(\mathbb{R}^d)$ with $\vert q' \vert <  \sqrt{2} -1$ or $\vert q' \vert \leq  d^{- \frac{1}{2}}$.
  \end{thm}
  \begin{proof}
  If $A_q(\RH)$ were to be isomorphic to $A_{q'}(\mathbb{R}^d)$ then the unique trace property of Theorem \ref{Thm=Simple} shows that the pair $(M_q(\RH), A_q(\RH))$ is isomorphic to  $(M_{q'}(\mathbb{R}^d), A_{q'}(\mathbb{R}^d))$, see \cite[Lemma 1.1]{CKL} for the standard argument. However this is not the case by Theorem \ref{Thm=AntiAOGaussian} and the fact that $(M_{q'}(\mathbb{R}^d), A_{q'}(\mathbb{R}^d))$ has   AO by \cite{CIW}, \cite{Shlyakhtenko} (the property AO$^+$ in these references directly implies AO). The argument for the non-isomorphism of  $A_q(\RH)$ and  $A_0(\RH)$ is the same where we use Theorem \ref{Thm=AntiAOGaussian} and Proposition \ref{Prop=AO} instead.
  \end{proof}
 
\begin{rmk}
Fix a real Hilbert space $\RH$  with $\dim(\RH) < \infty$ and complexification $H$ as before.
We call the C$^\ast$-subalgebra of $\mathcal{B}(F_q(H))$ generated by $l_q(\xi), \xi \in H$ the $q$-CCR algebra.  Shortly after completion of this paper it was announced in \cite{Kuzmin} that, for $\RH$ fixed, all $q$-CCR algebras for $-1 < q < 1$ are isomorphic. In particular these C$^\ast$-algebras are nuclear. Following the proof of \cite[Theorem 4.2]{Shlyakhtenko} while using that $\RH$ is finite dimensional, it follows that  $(M_{q'}(\RH), A_{q'}(\RH))$ has AO for all $-1 < q' < 1$. Consequently, Theorem \ref{Thm=Main} holds for any $-1 < q' < 1$. This also completely classifies when $q$-Gaussian von Neumann algebras have AO with respect to the underlying $q$-Gaussian C$^\ast$-algebra.
\end{rmk} 

\begin{rmk}
In principle it is possible to give a purely C$^\ast$-algebraic proof of Theorem \ref{Thm=Main} as well by considering the following version of AO. We say that a C$^\ast$-algebra has C$^\ast$AO if it has a unique faithful tracial state $\tau$ and the map $A \otimes_{{\rm alg}} A^{\op} \rightarrow \cB(L_2(A, \tau)) \slash \mathcal{K}(L_2(A, \tau)): a \otimes b^\op \mapsto a b^\op + \mathcal{K}(L_2(A, \tau))$ is continuous for the minimal tensor norm. Here $L_2(A, \tau)$ is the GNS-space for $\tau$ and $b^\op$ the right multiplication with $b$.   This property distinguishes the algebras then.
\end{rmk}

\begin{rmk}
The question stays open whether for a real infinite dimensional Hilbert space $\RH$ one can distinguish the von Neumann algebra $M_0(\RH)$ from  $M_q(\RH)$ with $-1 < q <  1, q \not = 0$.
\end{rmk}

\end{document}